\documentclass{amsart}
\usepackage{amsmath,amsfonts,euscript,amscd,amsthm,amssymb,upref,graphics}
\usepackage{mathrsfs}

\usepackage[all]{xy}

\usepackage{latexsym}
\usepackage{amsfonts,amssymb} 
\usepackage{graphicx}
\usepackage{tikz}
\usepackage{caption}
\usepackage{subcaption}
\usepackage{enumitem}

\usepackage{mathrsfs}

\newcommand{\field}[1]{\mathbb{#1}}
\newcommand{\A}{\field{A}}
\newcommand{\C}{\field{C}}

\newcommand{\N}{\field{N}}

\newcommand{\Z}{\field{Z}}

\theoremstyle{plain}

\newtheorem{theorem}{Theorem}[section]
\newtheorem{proposition}[theorem]{Proposition}
\newtheorem{lemma}[theorem]{Lemma}
\newtheorem{corollary}[theorem]{Corollary}
\newtheorem{definition}[theorem]{Definition}

\theoremstyle{definition}

\theoremstyle{remark}

\begin{document}

% Redefine "plain" pagestyle
\makeatletter	   % `@' is now a normal "letter' for LaTeX
\makeatother     % `@' is restored as a "non-letter" character

\title[Smooth Factorial affine surfaces of Log-Kodaira dimension 0 with trivial units]{Smooth Factorial affine surfaces of Logarithmic Kodaira dimension zero with trivial units}
\author{Gene Freudenburg \and Hideo Kojima \and Takanori Nagamine}
\date{\today} 
\thanks{The work of the second author was supported by JSPS KAKENHI Grant Number JP17K05198}
\thanks{The work of the third author was supported by a Grant-in-Aid for JSPS Fellows  (No. 18J10420) from the Japan Society for the Promotion of Science}

\begin{abstract} This paper considers the family $\mathscr{S}_0$ of smooth affine factorial surfaces of logarithmic Kodaira dimension 0 with trivial units over an algebraically closed field $k$. 
Our main result ({\it Theorem\,\ref{main}}) is that the number of isomorphism classes represented in $\mathscr{S}_0$ is at least countably infinite. This contradicts the earlier classification of Gurjar and Miyanishi \cite{Gurjar.Miyanishi.88} which asserted that $\mathscr{S}_0$ has at most two elements up to isomorphism when $k=\C$. Thus, the classification of surfaces in $\mathscr{S}_0$ for the field $\C$, long thought to have been settled, is an open problem. 
\end{abstract}

\maketitle

%%%%%%%%%%%%%%%%%%%%%%%%%%%%%%%%%%%%%%%%%%%%%%%%%%%%%%%%%%%%%%%%%%%%%%%

\section{Introduction} 

Let $k$ be an algebraically closed field $k$, and let $\mathscr{S}$ denote the family of smooth affine factorial surfaces with trivial units over $k$. Then $\mathscr{S}$ is partitioned according to logarithmic Kodaira dimension:
\[
\mathscr{S}=\mathscr{S}_2\cup\mathscr{S}_1\cup\mathscr{S}_0\cup\mathscr{S}_{-\infty}
\]
where $S\in\mathscr{S}_d$ means $\bar{\kappa}(S)=d$. Surfaces in $\mathscr{S}_2$ are said to be of general type. 
It is known that $\mathscr{S}_{-\infty}$ consists of one element (up to isomorphism), namely, the affine plane $\A^2_k$; see \cite{Miyanishi.01}. 

In their 1988 paper \cite{Gurjar.Miyanishi.88}, Gurjar and Miyanishi studied elements of $\mathscr{S}$ of non-general type in the case $k=\C$. 
For $\mathscr{S}_0$, they claim that
``there are only two types of surfaces in this case'' (p.99). Specifically, Theorem 2 asserts that any element of $\mathscr{S}_0$ is isomorphic to one of the surfaces $S$ or $S'$, defined as follows.
Let $x,y$ be a system of coordinate functions in the complex plane $\C^2$. 
\begin{itemize}
\item $S$ is obtained by blowing up $\C^2$ at the points $(0,1)$ and $(1,0)$, and removing the proper transform of the degenerate conic defined by $xy=0$.
\item $S'$ is obtained by blowing up $\C^2$ the point $(1,1)$, and removing the proper transform of the smooth conic $xy=1$.
\end{itemize}
In this paper, we will show the following.
\begin{enumerate}
\item $S\cong_{\C}S'$ \, ({\it Theorem\,\ref{Takanori}})
\item For any algebraically closed field $k$, the number of isomorphism classes represented in $\mathscr{S}_0$ is at least countably infinite.
\end{enumerate}
In fact, even the number of {\it stable} isomorphism classes represented in $\mathscr{S}_0$ is at least countably infinite (see {\it Theorem\,\ref{summary}}). 
Thus, the classification of surfaces in $\mathscr{S}_0$ for the field $\C$, long thought to have been settled, is an open problem. 

The surface $S'$ is the affine modification of the plane along the curve $xy=1$ with center $(1,1)$; see \cite{Kaliman.Zaidenberg.99}. 
For any algebraically closed field $k$, the curve $xy=1$ is isomorphic to the affine line over $k$ with one point removed, denoted $\A^1_*$. 
But this is just one of many ways to embed $\A^1_*$ in $\A^2_k$, and each of these embeddings, together with a point on the curve, yields an affine modification $X$ of $\A^2_k$. 
{\it Theorem\,\ref{lemma1}} gives conditions on the embedding which imply that $X\in\mathscr{S}$, and {\it Theorem\,\ref{Kodaira}} gives further conditions which imply $X\in\mathscr{S}_0$. 

We consider the family of embeddings $x^ny=1$, $n\ge 1$.
Specifically, let $V_n$ be the affine modification of $\A^2_k$ along the curve $\Gamma_n$ defined by $x^ny=1$ with center $(1,1)$. 
We show that $V_n\in\mathscr{S}_0$ for each $n\ge 1$, and that $V_m\not\cong V_n$ if $m\ne n$, and more generally, for cylinders over these surfaces, $V_m\times\A^d_k\not\cong V_n\times\A^d_k$ if $m\ne n$ and $d\ge 0$; see {\it Theorem\,\ref{summary}}. 

For each $n$,  we have $\Gamma_n\cong\A^1_*$, but it is not hard to show that the complements $\A^2_k\setminus\Gamma_m$ and $\A^2_k\setminus\Gamma_n$ are not isomorphic when $m\ne n$. Consequently, there is no automorphism of $\A^2_k$ transforming $\Gamma_m$ to $\Gamma_n$ when $m\ne n$. 
For the field $k=\C$, the closed embeddings of $\C^*$ in $\C^2$ have been classified (see 
\cite{Cassou-Nogues.Koras.Russell.09}), thus providing a rich family of surfaces to study, obtained by affine modification of $\C^2$. 

%%%%%%%%%%%%%%%%%%%%%%%%%%%%%%%%%%%%%%%%%%%%%%%%%%%%%%%%%%%%%%%%%%%%%%%
\section{Preliminaries}

\subsection{Some Notation} If $R$ is an integral domain, then ${\rm frac}(R)$ is the field of fractions of $R$. Given $n\in\N$, $R^{[n]}$ is the polynomial ring in $n$ variables over $R$, and $R^{[\pm n]}$ is the ring of Laurent polynomials in $n$ variables over $R$. If $S$ is an integral domain containing $R$, then ${\rm tr.deg}_RS$ is the transcendence degree of 
${\rm frac}(S)$ over ${\rm frac}(R)$. 

Assume that $k$ is algebraically closed. Given distinct $c_1,\hdots ,c_n\in k$, let 
\[
\A^1_{*n}(k)=\A^1_{*n}={\rm Spec}\, k[t,(t-c_1)^{-1},\hdots ,(t-c_n)^{-1}]
\]
i.e., an affine line over $k$ with $n$ points removed. If $n=1$, we also write $\A^1_*$. Given $d\in\N$, $\A^d_k$ denotes affine space over $k$ of dimension $d$. 

\subsection{The Degree-Neutral Invariant} 
Let $R$ be an integral domain, and
let $\deg :R\to \Z\cup\{ -\infty\}$ be a degree function. We say $\deg$ is  
{\bf non-negative} if $\deg$ is of the form $\deg :R\to\N\cup\{-\infty \}$, and {\bf trivial} if $\deg (R\setminus\{ 0\})=\{ 0\}$.
The induced filtration is
\[
R=\bigcup_{d\in\Z}\mathcal{F}_d
\]
where the sets $\mathcal{F}_d=\{ r\in R\, |\, \deg r\le d\}$ are the associated {\bf degree modules}. 
Note that $\deg$ can be extended to $K={\rm frac}(R)$ by letting $\deg (f/g)=\deg f-\deg g$ for $f,g\in R$, $g\ne 0$. 
Note also that, if $R$ is a field, then $\deg$ is a degree function on $R$ if and only if $(-\deg)$ is a discrete valuation of $R$. 

Recall that a subalgebra $A\subset R$ is {\bf factorially closed} in $R$ if $rs\in A$ for nonzero $r,s\in R$ implies $r\in A$ and $s\in A$.

\begin{proposition}\label{degree} With the assumptions and notation above:
\begin{itemize}
\item [{\bf (a)}] $\mathcal{F}_0$ is a subring of $R$ which is integrally closed in $R$.
\item [{\bf (b)}] $\mathcal{F}_d$ is an ideal of $\mathcal{F}_0$ for each $d\le 0$. 
\item [{\bf (c)}] If $\deg$ is non-negative, then $\mathcal{F}_0$ is factorially closed in $R$. Consequently, $R^*\subset\mathcal{F}_0$. 
\item [{\bf (d)}] If $R$ is a normal ring, then $\mathcal{F}_0$ is a normal ring. 
\item [{\bf (e)}] If $R$ is a field, then $\mathcal{F}_0$ is a valuation ring of $R$ and ${\rm frac}(\mathcal{F}_0)=R$. 
\end{itemize} 
\end{proposition} 

\begin{proof} Extend $\deg$ to $K$ and let $V=\{ f\in K\,\vert\, \deg f\le 0\}$. Then $V$ is a valuation ring of $K$, and $\mathcal{F}_0=V\cap R$. This implies parts (a), (d) and (e), and part (b) is clear.
For part (c), if $ab\in\mathcal{F}_0$ for nonzero $a,b\in R$, then $0=\deg (ab)=\deg a+\deg b$ in $\N$ implies $\deg a=\deg b=0$, so $a,b\in\mathcal{F}_0$.
\end{proof}

\begin{definition} {\rm The {\bf degree-neutral invariant} of $R$ is the subring of $R$ defined by the intersection of all subrings of the form $\mathcal{F}_0$ for some non-negative degree function on $R$, denoted by $\Delta^0 (R)$. We say that $R$ is {\bf degree-rigid} if $\Delta^0(R)=R$, i.e., every non-negative degree function on $R$ is trivial. }
\end{definition}

The degree-neutral invariant has the following properties. 
\begin{enumerate}
\item $\Delta^0(R)$ is a factorially closed subring of $R$, being the intersection of factorially closed subrings, and $\Delta^0(R)$ is an invariant subring of $R$. 
Therefore, $\Delta^0(R)$ is an algebraically closed subring of $R$, and if $R$ is a UFD, then $\Delta^0(R)$ is a UFD. 
\item $R^*\subset\Delta^0(R)$. 
\item If $S\subset R$ is a subring, then $\Delta^0(S)\subset\Delta^0(R)$, since any non-negative degree function on $R$ restricts to a non-negative degree function on $S$. 
\item If $R[x_1,\hdots ,x_n]\cong R^{[n]}$, then by considering the degree functions $\deg_{z_i}$, $1\le i\le n$, it is clear that $\Delta^0\left( R^{[n]}\right)\subset R$.
\item The Makar-Limanov invariant $ML(R)$ is defined by the intersection of subrings of the form $\mathcal{F}_0$ for certain non-negative degree functions; see \cite{Crachiola.06b,Freudenburg.17}. Therefore, $\Delta^0(R)\subset ML(R)$. $R$ is said to be rigid if $ML(R)=R$. Thus, degree-rigid implies rigid. 
\end{enumerate}

Recall that $R$ is {\bf strongly invariant} if the following condition holds (see \cite{Abhyankar.Eakin.Heinzer.72}). 
\begin{quote}
For any integer $n\ge 0$ and subring $S\subset R[x_1,\hdots ,x_n]\cong R^{[n]}$, if there exist $y_1,\hdots ,y_n$ in $R^{[n]}$ such that 
$R[x_1,\hdots ,x_n]=S[y_1,\hdots ,y_n]$, then $R=S$. 
\end{quote}

\begin{lemma}\label{deg-rig} Suppose that $R$ is degree-rigid.
\begin{itemize}
\item [{\bf (a)}] $\Delta^0(R^{[n]})=R$
\item [{\bf (b)}] $R$ is strongly invariant.
\end{itemize}
\end{lemma}

\begin{proof} Let $\delta$ be a non-negative degree function on $R^{[n]}$ with induced filtration $\bigcup_{i\in\N}\mathcal{F}_i$. 
Then $\delta$ restricts to $R$, so $\delta (f)=0$ every $f\in R\setminus\{ 0\}$, and $R\subset\mathcal{F}_0$. Therefore, $R\subset\Delta^0(R^{[n]})\subset R$, which implies
$\Delta^0(R^{[n]})=R$. This shows part (a).

For part (b), assume that $B$ is a ring containing $R$, and $B\cong_RR^{[n]}$ for some $n\ge 0$. Let $S\subset B$ be a subring such that $B\cong_SS^{[n]}$. 
Then $R=\Delta^0(B)\subset S\subset R^{[n]}$. Therefore, ${\rm tr.deg}_RS\le n$ and:
\[
n={\rm tr.deg}_RB={\rm tr.deg}_RS+{\rm tr.deg}_SB={\rm tr.deg}_RS+n \implies {\rm tr.deg}_RS=0
\]
Since both $S$ and $R$ are algebraically closed in $B$, it follows that $R$ is algebraically closed in $S$. Therefore, $R=S$.
\end{proof}

\subsection{Samuel's Criterion}
A well known criterion for a ring to be a UFD is given by Samuel \cite{Samuel.64}, Proposition 7.6.
In an integral domain $A$, elements $a,b\in A$ are {\bf relatively prime} if $aA\cap bA=abA$.

\begin{proposition}\label {Samuel-lemma}  Assume that $A$ is an integral domain and $a,b\in A$ are relatively prime and nonzero. Let $A[X]\cong A^{[1]}$ and $A'=A[X]/(aX-b)$. 
\begin{itemize}
\item [{\bf (a)}] $aX-b$ is a prime element of $A[X]$. 
\item [{\bf (b)}] If $A$ is a noetherian UFD and $aA$ and $aA+bA$ are prime ideals of $A$, then $A'$ is a noetherian UFD.
\end{itemize}
\end{proposition}

%%%%%%%%%%%%%%%%%%%%%%%%%%%%%%%%%%%%%%%%%%%%%%%%%%%%%%%%%%%%%%%%%%%%%%%

\section{Surfaces defined by $\A^1_*$ Embeddings}

In this section, we use certain closed embeddings of the punctured line $\A^1_*$ in the plane $\A^2_k$ to define surfaces belonging to $\mathscr{S}$. 

\subsection{A Family of Surfaces in $\mathscr{S}$}

\begin{theorem}\label{lemma1} Let $A=k[x,y]\cong k^{[2]}$ and let $a,b\in A$ be such that:
\[
A/aA\cong k^{[\pm 1]}\,\, ,\,\, A/bA\cong k^{[1]}\,\, ,\,\, A/(aA+bA)\cong k
\]
Define the ring $B\subset A_a$ by $B=A[b/a]$. 
\begin{itemize}
\item [{\bf (a)}] $B$ is a UFD and $B^*=k^*$.
\item [{\bf (b)}] If $k$ is algebraically closed and $X={\rm Spec}(B)$, then $X\in\mathscr{S}$. 
\end{itemize}
\end{theorem}

\begin{proof} We have $B\cong A[X]/(aX-b)$, where $A[X]\cong A^{[1]}$. In addition, $a$ and $b$ are prime elements of $A$, $\gcd_A(a,b)=1$ and the ideal $aA+bA$ is prime in $A$. By Samuel's Criterion, it follows that $B$ is a UFD. 
Since $A\subset B\subset A_a$ and $B\ne A_a$, we see that $B^*=A^*=k^*$. This proves part (a).

To prove part (b), let $\Gamma\subset\A^2$ be the curve defined by $a=0$, noting that $\Gamma\cong\A^1_*$. 
Define the maximal ideal $I=(a,b)\subset A$, and let $P\in\A^2_k$ be the point defined by $I$, noting that $P\in\Gamma$. 
Let $W$ be the blow-up of $\A^2_k$ at $P$. Then $W$ is smooth. Let $\Gamma^{\prime}\subset W$ be the proper transform of $\Gamma$. Then $W\setminus\Gamma^{\prime}$ is smooth, and the coordinate ring of $W\setminus\Gamma^{\prime}$ is the affine modification:
\[
A[a^{-1}I]=A[a^{-1}(a,b)]=A[b/a]=B
\] 
Therefore, $X=W\setminus\Gamma^{\prime}$ is smooth, which implies $X\in\mathscr{S}$.
\end{proof}

\subsection{Logarithmic Kodaira Dimension}

Assume that the ground field $k$ is algebraically closed. Let $B=A[b/a]$ be a ring of the type described in {\it Theorem\,\ref{lemma1}}, that is, $A=k[x,y]\cong k^{[2]}$ and $a,b\in A$ are such that:
\[
A/aA\cong k^{[\pm 1]}\,\, ,\,\, A/bA\cong k^{[1]}\,\, ,\,\, A/(aA+bA)\cong k
\]
Let $X={\rm Spec}(B)$. By {\it Theorem\,\ref{lemma1}}, $X\in\mathscr{S}$. 
\begin{theorem}\label{Kodaira}  Let $\Gamma\subset\A^2_k$ be the curve defined by the ideal $aA$.
\begin{itemize} 
\item [{\bf (a)}] $\bar{\kappa}(X)=\bar{\kappa}(\A^2_k\setminus\Gamma )$
\item [{\bf (b)}] If $\Gamma$ has two places at infinity, then $\bar{\kappa}(X)\ge 0$. 
\item [{\bf (c)}] If the complement $\A^2_k\setminus\Gamma$ contains an open set isomorphic to $\A^1_*\times\A^1_*$, then $\bar{\kappa}(X)\le 0$.
\end{itemize}
\end{theorem}
In order to prove the theorem, we need the following lemma, which is well-known. 

\begin{lemma}\label{Kojima}
Let $V$ be a smooth projective surface defined over $k$ and let $C$ be a reduced curve on $V$. Let $\pi: W \to V$ be the blow-up at a point $P$ on $C$ and let $D$ be the proper transform of $C$ on $W$. If $P$ is a smooth point of $C$, then $\bar{\kappa}(W \setminus D) = \bar{\kappa}(V \setminus C)$. 
\end{lemma}

\begin{proof}
We may assume that $C$ is an SNC-divisor, i.e., $C$ has only normal crossings and consists of smooth curves. Since $P$ is a smooth point of $C$, we have $K_W + D \sim \pi^*(K_V + C)$. So $H^0(W, \mathcal{O}_W(n(K_W+D))) \cong H^0(V, \mathcal{O}_V(n(K_V+C)))$ for any integer $n \geq 1$. Hence, $\bar{\kappa}(W \setminus D) = \bar{\kappa}(V \setminus C)$. 
\end{proof}

\begin{proof}[Proof of Theorem\,\ref{Kodaira}] Let $S=\A^2_k\setminus\Gamma$, let $P\in\A^2_k$ be the point defined by the maximal ideal $(a,b)$, let $\pi: W \to \A^2_k$ be the blow-up at the point $P$, and let $\Gamma^{\prime}\subset W$ be the proper transform of $\Gamma$. Then $X\cong W \setminus\Gamma^{\prime}$. Since $P$ is a smooth point of $\Gamma$, we infer from {\it Lemma\,\ref{Kojima}} that 
\[
\bar{\kappa}(X) = \bar{\kappa}(W  \setminus\Gamma^{\prime}) = \bar{\kappa}(S) 
\]
This proves part (a). 

For part (b), since $\Gamma$ is a smooth affine plane curve with two places at infinity, the log geometric genus of $S$ is positive; see, e.g., Chapter 2, Lemma 2.2.2 (p.\ 72) of \cite{Miyanishi.01}. 
Hence, $\bar{\kappa}(S) \geq 0$, so $\bar{\kappa}(X)\ge 0$. 

For part (c), since $S$ contains an open set $U\cong\A^1_*\times\A^1_*$, we have $\bar{\kappa}(S) \leq \bar{\kappa}( \A^1_* \times \A^1_*) =0$. Therefore, $\bar{\kappa}(X) \le  0$. 
\end{proof}

\section{The Rings $B_n$} 

\subsection{Definition} Let $A=k[x,y]\cong k^{[2]}$. Given $n\in\N$, $n\ge 0$, define the ring:
\[
B_n=A[(x-1)/(x^ny-1)]
\]
By {\it Theorem\,\ref{lemma1}}, we see that $B_n$ is a UFD and $B_n^*=k^*$. Let $u=(x-1)/(x^ny-1)$. 

\subsection{Fibrations} Assume that $k$ is algebraically closed, and fix the positive integer $n$. Let $V_n={\rm Spec}(B_n)$.
Given $\lambda\in k$, consider the fibers $x-\lambda$, $u-\lambda$ and $y-\lambda$. There are three cases.

\begin{enumerate}
\item General fiber: $\lambda\not\in\{ 0,1\}$
\subitem (i) ${\rm Spec}\, B/(x-\lambda)\cong {\rm Spec}\, k[y,u]/(u(\lambda^ny-1)+(1-\lambda))\cong \A^1_*$
\subitem (ii) ${\rm Spec}\, B/(u-\lambda)\cong {\rm Spec}\, k[x,y]/(\lambda(x^ny-1)-(x-1))\cong\A^1_*$
\subitem (iii) ${\rm Spec}\, B/(y-\lambda)\cong {\rm Spec}\, k[x,u]/(u(\lambda x^n-1)-(x-1))\cong\A^1_{*n}$
\smallskip
\item Reducible fiber: $\lambda =1$
\subitem (i) ${\rm Spec}\, B/(x-1)\cong {\rm Spec}\, k[y,u]/(u(y-1))\cong \A^1_k\cup\A^1_k$
\subitem (ii) ${\rm Spec}\, B/(u-1)\cong {\rm Spec}\, k[x,y]/(x(x^{n-1}y-1))\cong\A^1_k\cup\A^1_*$
\subitem (iii) ${\rm Spec}\, B/(y-1)\cong {\rm Spec}\, k[x,u]/((x-1)(u(x^{n-1}+\cdots +x+1)-1))\cong\A^1_k\cup\A^1_{*(n-1)}$
\smallskip
\item Zero fiber: $\lambda =0$
\subitem (i) ${\rm Spec}\, B/(x)\cong {\rm Spec}\, k[y]\cong \A^1_k$
\subitem (ii) ${\rm Spec}\, B/(u)\cong {\rm Spec}\, k[y]\cong\A^1_k$
\subitem (iii) ${\rm Spec}\, B/(y)\cong {\rm Spec}\, k[x]\cong\A^1_k$
\end{enumerate}

Note that the inclusions $k[x]\to B_n$ and $k[u]\to B_n$ induce $\A^1_*$-fibrations of $V_n$. Note also that 
the ideals $(x-\lambda)$, $(u-\lambda )$ and $(y-\lambda )$ are prime if and only if $\lambda\ne 1$. 

\subsection{$k$-Algebra Isomorphism Classes}

\begin{theorem}\label{main} Given $m,n\in\N\setminus\{ 0\}$, if $B_m\cong_kB_n$, then $m=n$. 
\end{theorem}

\begin{proof} It suffices to assume that $k$ is algebraically closed. 

Assume that $B_m\cong_kB_n$ and that $n>m\ge 1$. Then there exists a subalgebra $B\subset k(x,y)\cong k^{(2)}$ and elements $u,X,Y,U\in B$ such that:
\[
B=k[x,y,u]=k[X,Y,U]\,\, ,\,\, u(x^my-1)=x-1 \,\, ,\,\, U(X^nY-1)=X-1
\]
Let $A=k[x,y]$ and $t=x^my-1$, noting that $A\cong k^{[2]}$ and $A\subset B\subset B_t=A_t$. Define a $\Z$-grading of $A_t$ by letting $x,y$ be homogeneous 
with $\deg x=1$ and $\deg y=-m$. It is easy to see that:
\begin{equation}\label{localization}
h\in A_t \,\,\, {\rm and}\,\, \deg h<0 \implies h\in yA_t
\end{equation}
The degree function restricts to $B$. Let $B=\bigcup_{d\in\Z}\mathcal{F}_d$ be the filtration of $B$ by degree modules. 
We claim that the following two properties hold:
\begin{equation}\label{degree}
h\in B \,\,\, {\rm and}\,\, \deg h<0 \implies h\in yB
\end{equation}
and
\begin{equation}\label{prime}
h\in B \,\,\, {\rm is\,\, prime \,\, and}\,\, \deg h<0 \implies hB=yB \,\, {\rm and}\,\, k[h]=k[y]
\end{equation}
Consequently, when condition (\ref{prime}) holds, the general fiber of $h$ is $\A^1_{*m}$ and $h$ has exactly one reducible fiber, that fiber being isomorphic to $\A^1_k\cup\A^1_{*(m-1)}$. 

To prove these implications, note that, since $B$ is a UFD, $y$ is prime in $B$ and $yB+tB=B$, it follows that $yB_t\cap B=yB$. Therefore, if $h\in B$ and $\deg h<0$, then $h\in yB_t$ by the implication (\ref{localization}). But then $h\in yB_t\cap B=yB$. So the implication (\ref{degree}) is confirmed. 

Suppose $h\in B$ is prime and $\deg h<0$. Then by (\ref{degree}) we have $hB=yB$, since $y$ is prime in $B$. Since $B^*=k^*$, it follows that $k[h]=k[y]$. 
So the implication (\ref{prime}) is confirmed.

Since $Y$ is prime and the general fiber of $Y$ is $\A^1_{*n}$, and since $\A^1_{*n}\not\cong\A^1_{*m}$, it follows that $\deg Y\ge 0$. 

In addition, $U$ is prime, the general fiber of $U$ is $\A^1_*$, and the reducible fiber of $U$ is $\A^1_k\cup\A^1_*$. If $\deg U<0$, then $\A^1_*=\A^1_{*m}$ implies $m=1$. But then the reducible fibers give $\A^1_k\cup\A^1_k\cong\A^1\cup\A^1_*$, a contradiction. So $\deg U\ge 0$. 

Since $X^nY-1$ is prime and has a reducible fiber with $n+1>2$ components, it follows that $\deg (X^nY-1)\ge 0$. Since $X-1=U(X^nY-1)$, we also see that
$\deg (X-1)\ge 0$. 

There are 2 cases to consider.
\medskip

\noindent {\bf Case 1:} $\deg X\ge 0$. Then
\[
\deg X=\deg (X-1+1)\le\max\{ \deg (X-1),0\}=\deg (X-1)\le\max\{ \deg X,0\}=\deg X 
\]
which implies $\deg X=\deg (X-1)$. The same reasoning shows $\deg (X^nY)=\deg (X^nY-1)$. 
Therefore:
\begin{eqnarray*}
\deg X &=& \deg (X-1) \\
&=& \deg(U(X^nY-1)) \\
&=& \deg U + \deg (X^nY-1) \\
&=& \deg U + \deg (X^nY) \\
&=& \deg U + n\deg X +\deg Y
\end{eqnarray*}
Since $\deg X,\deg Y,\deg U\in\N$ and $n\ge 2$, this gives:
\[
0=\deg U+(n-1)\deg X+\deg Y \implies \deg X=\deg Y=\deg U=0
\]
But then
\[
k[X,Y,U]\subset \mathcal{F}_0\subset B=k[X,Y,U] \implies \mathcal{F}_0=B
\]
which is a contradiction, since $x\not\in\mathcal{F}_0$. So the case $\deg X\ge 0$ cannot occur.
\medskip

\noindent {\bf Case 2:} $\deg X< 0$. Then 
(3) implies $XB=yB$ and $k[X]=k[y]$. So $X=cy$ for some $c\in k^*$. Since the general fiber of $X$ is $\A^1_*$ and the general fiber of $y$ is $\A^1_{*m}$, it follows that $m=1$. 

Let $T=X^nY-1$, $K=k(X)=k(y)$ and $B_K=K\otimes_kB$. Since the generic fiber of $X=cy$ is $\A^1_*(K)$, we see that $B_K\cong K^{[\pm 1]}$. Specifically:
\[
B_K=K[x,y,u]=K[x,u]=K[y^{-1}(t+1),(y^{-1}-1)t^{-1}+y^{-1}]=K[t,t^{-1}]
\]
and 
\[
B_K=K[X,Y,U]=K[Y,U]=K[y^{-n}(T+1),(y-1)T^{-1}]=K[T,T^{-1}]
\]
It follows that $aT=bt^{\pm 1}$ for some $a,b\in k[y]$ with $\gcd (a,b)=1$. 

Assume $aT=bt$. Since $\gcd (a,b)=1$, we see that either both $a$ and $b$ are prime in $B$, or $a,b\in B^*=k^*$. 
The case $a,b\in k^*$ is not possible, since the reducible fiber of $T$ has $n+1>2$ components, and the reducible fiber of $t$ has $m+1=2$ components. Therefore, both $a$ and $b$ are prime in $B$. 
But then $aB=tB$ implies $t=ra\in K$ for some $r\in k^*$, a contradiction, since $t$ is transcendental over $K$. 

Assume $aTt=b$. Then $b\in aB$ implies $a\in k^*$, and $b$ has exactly two prime factors. One possibility is that $b$ has a root $\lambda\in k\setminus\{ 1\}$. But then either $T=r(y-\lambda)\in K$ or $t=r(y-\lambda )\in K$ for $r\in k^*$, which is a contradiction. The other possibility is that $b=s(y-1)$ for $s\in k^*$. The equality $Tt=a^{-1}s(y-1)$ shows $a^{-1}s=-1$. Since $X=cy$, we have:
\begin{eqnarray*}
Tt=(X^nY-1)(xy-1)=1-y &\implies&  X^nYxy-X^nY-xy=-y \\
&\implies& X^nYx-cX^{n-1}Y-x=-1
\end{eqnarray*}
Since $n\ge 2$, we conclude that $xB+yB=B$, a contradiction since $xB+yB$ is a maximal ideal of $B$.
So the case $\deg X<0$ cannot occur. 

The assumption that $m\ne n$ thus leads to contradiction. Therefore, $m=n$. 
\end{proof}

\subsection{Degree-Rigidity of $B_n$}

\begin{theorem}\label{non-neg} $B_n$ is degree-rigid for each $n\ge 1$. 
\end{theorem}

\begin{proof} Let $B_n=k[x,y,(x-1)/(x^ny-1)]$, let $\delta :B_n\to\N\cup\{-\infty \}$ be a non-negative degree function, and let $B_n=\bigcup_{i\in\N}\mathcal{F}_i$ be the filtration induced by $\delta$. 
Given $\lambda ,\mu\in k$ and $f\in B_n\setminus k$, since $k^*\subset\mathcal{F}_0$, we have:
\[
0\le \delta (f-\lambda )=\delta ((f-\mu)+(\mu -\lambda))\le\max\{ \delta (f-\mu), \delta (\lambda -\mu)\}=\delta (f-\mu )
\]
Therefore, $\delta (f-\lambda)=\delta (f-\mu )$ for all $f\in B\setminus k$ and $\lambda ,\mu\in k$. It follows that
\[
0\le \delta \left(\frac{x-1}{x^ny-1}\right) =\delta (x-1)-\delta (x^ny-1) = \delta (x)-\delta (x^ny)=\delta (x)-n\delta (x)-\delta (y)
\]
which gives:
\[
0\le (n-1)\delta (x)\le -\delta (y) \implies (n-1)\delta (x)=\delta (y)=0
\]
If $n\ge 2$, we see that $\delta (x)=0$. 
If $n=1$, there is an automorphism $\alpha$ of $B_1$ such that $\alpha (y)=x$. Therefore, $\delta (x)=0$ if $n=1$. 

So in all cases, $\delta (x)=\delta (y)=0$, meaning that $k[x,y]\subset\mathcal{F}_0$. By {\it Proposition\,\ref{degree}(c)}, $\mathcal{F}_0$ is factorially closed, and thus algebraically closed, in $B_n$. Therefore, the algebraic closure of $k[x,y]$ in $B_n$ is contained in $\mathcal{F}_0$, which implies $B_n\subset\mathcal{F}_0$, i.e., $B_n=\mathcal{F}_0$.
\end{proof}

\begin{corollary}\label{d-variables} Given integers $m,n,d\ge 1$, if $B_m^{[d]}\cong_k B_n^{[d]}$, then $m=n$. 
\end{corollary}

\begin{proof} By {\it Theorem\,\ref{non-neg}}, $B_m$ is degree-rigid. Therefore, by {\it Lemma\,\ref{deg-rig}}, $B_m$ is strongly invariant. 
Since $B_m^{[d]}\cong_k B_n^{[d]}$, we have $B_m\cong_kB_n$. 
By {\it Theorem\,\ref{main}}, it follows that $m=n$.
\end{proof}

\subsection{Coordinate Rings for Two Surfaces of Gurjar and Miyanishi} 
Let $k$ be a field. We consider the following two $k$-algebras $C_1, C_2$: 
\[
C_1=k[x,y,u,v]/(ux-(y-1),vy-(x-1))
\]
and
\[
C_2=k[X,Y,U,V]/(U(XY-1)-(X-1),V(XY-1)-(Y-1))
\]
Observe that, for the surfaces $S$ and $S'$ defined in the {\it Introduction}, we have $S={\rm Spec}(C_1)$ and $S'={\rm Spec}(C_2)$ when $k=\C$. 

\begin{theorem}\label{Takanori} For any field $k$, the rings $C_1$, $C_2$ and $B_1$ are isomorphic as $k$-algebras.
\end{theorem}
\begin{proof}
Let $B=k[x,y,u,v]\cong k^{[4]}$ and define ideals:
\[
I=(ux-(y-1),vy-(x-1)) \,\, ,\,\, J=(x(uv-1)+(v+1), y(uv-1)+(u+1))
\]
First of all, we show $I=J$, that is, $C_1=B/J$. 

In $B/I$, we have: 
\begin{eqnarray*}
x(uv-1)+(v+1) &=& v(ux+1)-(x-1) \\
&=& vy-(x-1) \\
&=& 0
\end{eqnarray*}
Similarly we have that $y(uv-1)+(u+1)=0$ in $B/I$. Therefore $J\subset I$. 

On the other hand, in $B/J$, we have: 
\begin{eqnarray*}
(uv-1)(1+ux) &=& uv-1+u(x(uv-1)) \\
&=& uv-1-u(v+1) \\
&=& -(u+1)\\
&=& y(uv-1)
\end{eqnarray*}
Hence $1+ux=y$ in $B/J$. Similarly we have that $1+vy=x$ in $B/J$. Therefore, in $B/J$,  we have 
\begin{eqnarray*}
ux-(y-1) &=& u(1+vy)-(y-1) \\
&=& y(uv-1)+(u+1) \\
&=& 0
\end{eqnarray*}
and $vy-(x-1)=0$, which implies $I\subset J$. Hence $I=J$. 

Define the $k$-isomorphism $\varphi:B\to k[X,Y,U,V]\cong k^{[4]}$ by $\varphi (x,y,u,v)=(U,V,-Y,-X)$. Then: 
\[
\varphi((x(uv-1)+(v+1), y(uv-1)+(u+1))=(U(XY-1)-(X-1), V(XY-1)-(Y-1))
\]
Therefore we have $C_1\cong_k\varphi(B/J)=C_2$. 

Finally, in $C_2$ we have 
\[
\frac{Y-1}{XY-1}=1-Y\frac{X-1}{XY-1}
\]
which shows $C_2=k[X,Y,\frac{X-1}{XY-1}]\cong B_1$.
\end{proof}

%%%%%%%%%%%%%%%%%%%%%%%%%%%%%%%%%%%%%%%%%%%%%%%%%%%%%%%%%%%%%%%%%%%%%%%

\section{Isomorphism Classes in $\mathscr{S}_0$} Assume that $k$ is algebraically closed. Given $n\ge 1$, let $V_n={\rm Spec}(B_n)$. 
As observed above, $B_n=A[b/a]$ satisfies the hypotheses of {\it Theorem\,\ref{lemma1}} for $A=k[x,y]$ and elements $a=x^ny-1$ and $b=x-1$. 
Therefore, $V_n\in\mathscr{S}$ for each $n\ge 1$. 

Let $\Gamma_n\subset\A^2_k$ be the curve defined by the ideal $aA$. Then $\Gamma_n$ has two places at infinity. By {\it Theorem\,\ref{Kodaira}(b)}, $\bar{\kappa}(V_n)\ge 0$.

Consider the localization of the ring $A_a$ at $x$:
\[
A_a[x^{-1}]=k[x,x^{-1},y,(x^ny-1)^{-1}]=k[x,x^{-1},x^ny-1,(x^ny-1)^{-1}]\cong k^{[\pm 2]}
\]
Therefore, $\A^2_k\setminus\Gamma_n$ contains an open subset isomorphic to $\A^1_*\times\A^1_*$. By {\it Theorem\,\ref{Kodaira}(c)}, $\bar{\kappa}(V_n)\le 0$. 

It follows that $\bar{\kappa}(V_n)=0$. Moreover, {\it Theorem\,\ref{main}} shows that $V_m\cong V_n$ implies $m=n$, and {\it Corollary\,\ref{d-variables}} shows that $V_m\times\A^d_k\cong V_n\times\A^d_k$ for some $d\ge 0$ implies $m=n$.
These results are summarized below. 
\begin{theorem}\label{summary} For the surfaces $V_n$, $n\ge 1$, the following properties hold. 
\begin{itemize}
\item [{\bf (a)}] $V_n\in\mathscr{S}_0$ for each $n\ge 1$
\item [{\bf (b)}] Given positive integers $m,n$, if $V_m\times\A^d_k\cong_kV_n\times\A^d_k$ for some $d\in\N$, then $m=n$.
\end{itemize}
\end{theorem}

\noindent {\bf Question.} Over the field $k=\C$, are the surfaces $V_n$, $n\ge 1$, pairwise analytically isomorphic (respectively, pairwise homeomorphic)? 
%%%%%%%%%%%%%%%%%%%%%%%%%%%%%%%%%%%%%%%%%%%%%%%%%%%%%%%%%%%%%%%%%%%%%%%

%\bibliography{bibfile}
%\bibliographystyle{amsplain}

\bigskip

\noindent \address{Department of Mathematics\\
Western Michigan University\\
Kalamazoo, Michigan 49008} USA\\
\email{gene.freudenburg@wmich.edu}
\bigskip

\noindent\address{Department of Mathematics\\
Faculty of Science\\
Niigata University\\
8050 Ikarashininocho, Niigata 950-2181, Japan \\
\email{kojima@math.sc.niigata-u.ac.jp} 
\bigskip

\noindent\address{Graduate School of Science and Technology\\
Niigata University\\
8050 Ikarashininocho, Niigata 950-2181, Japan \\
\email{t.nagamine14@m.sc.niigata-u.ac.jp}
%%%%%%%%%%%%%%%%%%%%%%%%%%%%%%%%%%%%%%%%%%%%%%%%%%%%%%%%%%%%%%%%%%%%%%%
\end{document}